\documentclass[10pt,twoside]{article}
\usepackage{amssymb}

\usepackage{amsmath}
\usepackage{latexsym}

\newcommand{\erw}{elephant random walk}
\newcommand{\cfn}{{\cal F}_n}

\newcommand{\cgn}{{\cal G}_n}

\newcommand{\ind}{ 1\hspace{-1mm}1}

\newcommand{\memn}{{\mathfrak M}_n}
\newcommand{\xnp}{X_{n+1}}

\newcommand{\COM}[1]{}

\newtheorem{theorem}{Theorem}[section]

\newtheorem{lemma}{Lemma}[section]
\newtheorem{proposition}{Proposition}[section]

\newtheorem{remark}{\normalfont\scshape Remark}[section]

\newenvironment{proof}{\noindent\textsc{Proof.\/}}{}

\newcommand{\subj}[2]{\textsf{AMS 2000 subject classifications.}
Primary {#1}; Secondary {#2}.\newline}
\newcommand{\key}[1]{\textsf{Keywords and phrases.} {#1}.\newline}
\newcommand{\abb}[1]{\textsf{Abbreviated title.} {#1}.}

\newcommand{\fot}[5]{\renewcommand\thefootnote{}
\footnotetext{\parindent=0.0mm \vskip-3mm \subj{#1}{#2}\key{#3}\abb{#4}
\newline\textsf{Date.} \date{\today}}}

\def\vsb{\hfill$\Box$}

\def\vsp{\vskip-8mm\hfill$\Box$\vskip3mm}

\newcommand{\veps}{\varepsilon}

\newcommand{\be}{\begin{equation}}
\newcommand{\ee}{\end{equation}}
\newcommand{\bea}{\begin{eqnarray}}
\newcommand{\eea}{\end{eqnarray}}
\newcommand{\beaa}{\begin{eqnarray*}}
\newcommand{\eeaa}{\end{eqnarray*}}
\newcommand{\beal}{\begin{aligned}}
\newcommand{\eeal}{\end{aligned}}

\newcommand{\var}{\mathrm{Var\,}}

\newcommand{\sumk}{\sum^n_{k=1}}

\newcommand{\ttt}[1]{\quad\mbox{ #1}\quad}

\newcommand{\asto}{\stackrel{a.s.}{\to}}

\newcommand{\dto}{\stackrel{d}{\to}}

\newcommand{\nifi}{n\to\infty}
\newcommand{\kifi}{k\to\infty}

\begin{document}
\date{}
\title{\textsf{The number of zeroes in Elephant random walks with delays}}
\author{Allan Gut\\Uppsala University \and Ulrich Stadtm\"uller\\
Ulm University}
\maketitle

\begin{abstract}\noindent
In the simple random walk the steps are independent, whereas in the \erw\ (ERW), which was introduced by Sch\"utz and Trimper in 2004 \cite{erwdef}, the next step always depends on the whole path so far. In an earlier paper  we investigated \erw s when the elephant has a restricted memory. Inspired by a suggestion by Bercu et al.\ \cite{bercu2} we extended our results to the case when delays are allowed.  In this paper we examine how  the number of delays (that possibly stop the process) increases as time goes by.
\end{abstract}

\fot{60F05, 60G50,}{60F15, 60J10}{Elephant random walk, delay, number of zeroes, law of large numbers, central limit theorem, Markov chain}
{ERW with delays}

\section{Introduction}
\setcounter{equation}{0}
\markboth{A.\ Gut and U.\ Stadtm\"uller}{Elephant random walk}

In the classical \emph{simple\/} random walk the steps are equal to plus or minus one and independent---$P(X=1)=1-P(X=-1)=p$, ($0<p<1$); the walker has no memory.  Motivated by applications, although interesting in its own right, is the so called \erw\ (ERW),  for which every step depends on the whole process so far. The ERW  was introduced in \cite{erwdef} in 2004, the name being inspired by the fact that elephants have a very long memory. In \cite{111} we studied the case when the elephant has a restricted memory; assuming that he or she remembers only some distant past, only a recent past, or a mixture of both. Inspired by a suggestion in \cite{bercu2} we allowed, in \cite{113}, the possibility of delays in that the elephant, in addition, always has a choice of staying put. 

Formally, the \erw\ is defined as a  random walk in which the first step $X_1$ equals 1 with probability $s\in [0,1]$ and  to $-1$ with probability $1-s$, where, for convenience, we assume that $s=p$. After $n$ steps, at position $S_n=\sumk X_k$, the next step is defined as
\bea\label{20}\xnp=\begin{cases} +X_{K},\ttt{with probability} p\in[0,1], \\-X_{K},\ttt{with probability} 1-p,\end{cases}\eea
where $K$ has a uniform distribution on the integers $1,2,\ldots,n$. 

In \cite{111} we studied the case of a restircted memory, which means that $K$ in (\ref{20}) is uniform over the set of points constituted by the memory, with the additional possibility of $0$ in \cite{113}. Letting $\memn$ denote the set of integers that constitute the memory up to time $n$, the rule for a next step is governed by
\bea\label{23}\xnp=\begin{cases} +X_{K},\ttt{with probability} p\in[0,1], \\-X_{K},\ttt{with probability} q\in[0,1],\\\phantom{+)}0,\ttt{\,\,  with probability}r\in[0,1],\end{cases}\eea
where  $p+q+r=1$, and where $K$ has a uniform distribution over the integers in $\memn$. In particular, if $\memn=\{1,2,\ldots,n\}$ the setting reduces to that suggested in \cite{bercu2}, and if, in addition, $r=0$ we are back in \cite{bercu}.
Let us here only mention that the evolution of the various  \erw s may differ considerably depending on the actual memory. For example if the memory consists of the most recent step only, the process stops as soon as a zero appears. In other cases one has a central limit theorem.

It is clear from the construction that the number of zeroes is monotone increasing (non-decreasing) to infinity. The number of ones is also increasing, although maybe not to infinity.  Motivated by a remark from Svante Janson we investigate, in this paper, the growth rate of zeroes for the various setups. Toward that end we define, for $n \in\mathbb{N}$, 
\bea\label{in}
I_n=\ind_{\{X_n=0\}}\ttt{with}N_n=\sum_{k=1}^n I_k.
\eea
Since it is natural to expect that $N_n\approx n$ in some sense as $n$ tends to infinity it is not surprising that it is mathematically more convenient to study how $N_n$ ''approaches'' $n$, that is, to investigate the difference
$n-N_n = $ the number of ones. We therefore also introduce
\bea\label{istern}
 I^*_n=\ind_{\{X_n\neq0\}}\ttt{with}N^*_n=\sum_{k=1}^n I^*_k.
\eea
Any result for the latter one can easily be transferred to the nonstarred one via the fact that $N_n+N^*_n =n$.

After some preliminaries in Section \ref{prel} we present, in subsequent sections, our results for the various memory models, after which we conclude with some remarks. In order to avoid special boundary effects we assume throughout that $0<p,q,r<1$. The standard $\delta_a(x)$ is used to  denote the distribution function with a jump of height one at $a$, $|A|$ denotes the cardinality of a set $A$, and $c$ and $C$ are numerical constants that may change between appearances.

\section{Preliminaries}
\setcounter{equation}{0}\label{prel}
\textbf{(i)}\quad In \cite{bercu} the behavior of the next step is governed by the relation $E(\xnp\mid \cfn) =(2p-1)\cdot\frac{S_n}{n}$,
since in that case $r=0$. The relation remains true if $r>0$ with $2p-1$ replaced by $p-q$. 

Our first tool is an analog for ERW:s with a restricted memory. Therefore, let  $\{\cfn,\,n\geq1\}$ denote the   $\sigma$-algebras generated by the memory  $\memn$, and let $\cgn= \sigma\{X_1,X_2,\ldots,X_n\}$.  Then,
\bea\label{uli}
E(\xnp\mid \cfn)
=(p-q)\cdot\frac{\sum_{i\in \memn}X_i}{|\memn|}.
\eea
When we condition on steps that are not contained in the memory it means that the elephant does not remember them, and, hence, cannot choose among them in a following step. Thus, if $A\subset \{1,2,\ldots,n\}$ is an arbitrary set of indices, such that $A\cap \memn=\emptyset$, then
\bea\label{uli2}
E(\xnp\mid \sigma\{A\cup \memn\})=E(\xnp\mid\cfn)=(p-q)\frac{\sum_{i\in \memn}X_i}{|\memn|}.
\eea

\noindent\textbf{(ii)}\quad
We also need some well-known facts about linear difference equations. 
\begin{proposition}\label{diff}
Consider the first order equation
\[x_{n+1}=a\, x_n +b_n \ttt{for} n \ge 1, \ttt{with} x^*_1 \ttt{given.}\]
Then
\[x_n=a^{n-1}x^*_1+\, \sum_{\nu=0}^{n-2}a^\nu b_{n-1-\nu}.\]
If, in addition, $|a|<1$ and $b_n=bn^\gamma$ with $\gamma>-1$, then
\[x_n=\frac{b_{n-1}}{1-a} -\frac{\gamma ab_{n-1}}{n(1-a)^2}\big(1+o(1)\big)\ttt{as}\nifi.\]
\end{proposition}

\noindent\textbf{(iii)}\quad
Next is a martingale lemma.
\begin{lemma}\label{HMG}
Let $\{U_n,\, n\geq1\}$ be a sequence of random variables adapted to $\cfn,\, n\geq1$, with
\[ E (U_{n+1} \mid\cfn)= a_n\,U_n +b_n \quad \mbox{ for } n \ge 1\]
with two squences $\{a_n\}$ and $\{b_n\}$,  $n\geq1$, where $a_n\not= 0$ for all $n$.
Then 
\[\{(M_n= \alpha_n \,U_n +\beta_n,\cfn)\,,n\geq1\}\ttt{is a martingale,}\]
  where  $\alpha_1=1$, $\beta_1=0$ and
\[ \alpha_n=\prod_{k=1}^{n-1} \frac{1}{a_k} \ttt{ and } \beta_n=-\sum_{k=1}^{n-1}\alpha_{k+1}\,b_k \ttt{for } n\ge 2\,.\]
\end{lemma}
The proof  amounts to checking that the martingale condition is satisfied. We omit the details.\medskip

\noindent\textbf{(iii)}\quad Finally some asymtotics related to the Gamma-function.    
\begin{lemma}\label{Gamma}
For $x \in \mathbb{R}$,
\[ \frac{\Gamma(n+1+x)}{\Gamma (n+1)} = n^x \Big( 1+\frac{x(1+x)}{2n} +{\cal O}(n^{-2})\Big)\ttt{as}\nifi.\]
\end{lemma}
\begin{proof} By Stirling's asymptotic formula,
\begin{eqnarray*}
\frac{\Gamma(n+1+x)}{\Gamma(n+1)}&=& \frac{\Big(\frac{n+x}{e}\Big)^{n+x}\sqrt{2\pi (n+x)}\Big(1+\frac{1}{12n}+{\cal O}(n^{-2})\Big)}{\Big(\frac{n}{e}\Big)^{n}\sqrt{2\pi (n)}\Big(1+\frac{1}{12n}+{\cal O}(n^{-2})\Big)}\\
&=& n^x e^{-x} (1+\frac{x}{n})^{n+x}\sqrt{1+\frac{x}{n}}\Big(1+\frac{1}{12n}- \frac{1}{12n}
+{\cal O}(n^{-2})\Big)\,.
\end{eqnarray*}
Next,
\begin{eqnarray*}
\Big(1+\frac{x}{n}\Big)^{n+x}&=&\exp\Big(\log(1+\frac{x}{n})\,(n+x)\Big)
=\exp\Big( \Big(\frac{x}n-\frac{x^2}{2n^2}+{\cal O}(n^{-2})\Big) (n+x)\Big)\\
&=&\exp\Big(x+ \frac{x^2}{n}-\frac{x^2}{2n}+{\cal O}(n^{-2})\Big)
=e^x \big(1+ \frac{x^2}{2n} +{\cal O}(n^{-2})\Big)\,.
\end{eqnarray*}
Thus,
\begin{eqnarray*}
\frac{\Gamma(n+1+x)}{\Gamma(n+1)}&=& n^x\Big(1+\frac{x}{2n} +{\cal O}(n^{-2})\Big)\Big(1+\frac{x^2}{2n}+{\cal O}(n^{-2})\Big)\\
&=&n^x\Big( 1+\frac{x\,(1+x)}{2n}+{\cal O}(n^{-2})\Big)\,.
\end{eqnarray*}\vsp
\end{proof}

\section{The case $\cfn =\sigma\{X_1,X_2,\ldots,X_n\}$}
\setcounter{equation}{0}
We first note that $E (I^*_1)=1-r$ and that $E(I^*_2)=(1-r)^2$. Furthermore,
\[ E (I^*_{n+1}\,\big|\, {\cal F}_n)=(
1-r) \cdot \frac{1}{n}\sum_{k=1}^n \ind_{ \{I^*_k\not=0 \} } = \frac{1-r}{n}N^*_n,\]
and thus
\[ E (N^*_{n+1}\,\big|\, {\cal F}_n)=\frac{1-r}{n}N^*_n+N^*_n,\]
so that
\[E(N^*_{n+1})=\prod_{k=1}^n \frac{k+1-r}{k} E(N^*_1)=\frac{1-r}{\Gamma(2-r)}\cdot \frac{\Gamma(n+1+(1-r))}{\Gamma(n+1)}= \frac{\Gamma(n+2-r)}{\Gamma(1-r)\Gamma(n+1)}\,.\]
Lemma \ref{Gamma} now tells us that
\bea\label{erwartung}
 E(N^*_{n+1})= \frac{n^{1-r}}{\Gamma(1-r)}+ n^{-r}\frac{(1-r)(2-r)}{2\Gamma(1-r)} +{\cal O}(n^{-r-1})\,.\eea
Next,
\begin{eqnarray*}
E((N^*_{n+1})^2 \, \big| \,{\cal F}_n)&=& (N^*_n)^2+E((I^*_{n+1})^2\,\big|\, {\cal F}_n)+2 N^*_n E(I^*_{n+1}\,\big|\, {\cal F}_n)\\
&=&E((N^*_n)^2)\cdot\Big(1+\frac{2(1-r)}{n}\Big) + \frac{1-r}{n} N^*_n,
\end{eqnarray*}
and thus
\begin{eqnarray*}
E((N^*_{n+1})^2)&=& E((N^*_n)^2)\cdot\Big(1+\frac{2(1-r)}{n}\Big) + \frac{1-r}{n} E(N^*_n)\\
&=&E((N^*_n)^2)\cdot\Big(1+\frac{2(1-r)}{n}\Big) + \frac{1-r}{n\,\Gamma(1-r)} \frac{\Gamma(n+(1-r))}{\Gamma(n)}\,.\\
\end{eqnarray*}
Induction shows that
\begin{eqnarray*}
\lefteqn{E((N^*_{n+1})^2)=}\\
&=&E((N^*_1)^2)\,\prod_{k=0}^{n-1} \Big(1+\frac{2(1-r)}{n-k}\Big) +\frac{(1-r)}{\Gamma(1-r)} \sum_{k=0}^{n-1} \frac{\Gamma(n-k+1-r)}{\Gamma(n+1-k)}\prod_{\mu=0}^{k-1}\Big(1+\frac{2(1-r)}{n-\mu}\Big)\\
&=&\frac{(1-r)}{\Gamma(3-2r)}\cdot\frac{\Gamma(n+1+2(1-r))}{\Gamma(n+1)}+\frac{(1-r)}{\Gamma(1-r)} \sum_{k=1}^{n} \frac{\Gamma(k+1-r)}{\Gamma(k+1)}\prod_{m=k+1}^{n}\Big(1+\frac{2(1-r)}{m}\Big)\\
&=&\frac{(1-r)}{\Gamma(3-2r)}\cdot\frac{\Gamma(n+1+2(1-r))}{\Gamma(n+1)}+\frac{1-r}{\Gamma(1-r)}\cdot \frac{\Gamma(n+1+2(1-r))}{\Gamma(n+1)}\sum_{k=1}^n \frac{\Gamma(k+(1-r))}{\Gamma(k+1+2(1-r))} \\
&=&\frac{\Gamma(n+1+2(1-r))}{\Gamma(n+1)}\cdot\Big( \frac{1-r}{\Gamma(3-2r)} +\frac{1-r}{\Gamma(1-r)}c_r (1+o(1)) \Big)\\
&=&d_r\,n^{2(1-r)} +o(n^{2(1-r)})\ttt{as}\nifi,\eeaa
where (observe Lemma \ref{Gamma})
\[ c_r=\sum_{k=1}^\infty \frac{\Gamma(k+1-r)}{\Gamma(k+1-r +2-r)}<\infty \ttt{ and } 
d_r=\frac{1-r}{\Gamma(1-r)} \Big(c_r + \frac{\Gamma(1-r)}{2(1-r)\Gamma(2(1-r))}\Big)\,.
\]
Using Lemma \ref{HMG} we find that
\begin{equation}
\{M^*_n=\alpha^*_n N^*_n,\,n\geq1\} \ttt{where} \alpha^*_k=\prod_{j=1}^{k-1} \frac{j}{j+1-r}\quad (k\geq1), 
\end{equation}
is a martingale. Obviously $E (M^*_1) =E (I^*_1)=1-r= E (M^*_n)$ for all $n \in \mathbb{N}$.

Next observe that 
\bea\label{alfastern}
\alpha^*_k=\prod_{j=1}^{k-1}  \frac{j}{j+1-r}=\frac{\Gamma(k) \Gamma(2-r)}{\Gamma(k+1-r)} \sim \frac{\Gamma(2-r)}{k^{1-r}}\ttt{as}\kifi.
\eea
For the martingale difference sequence $\veps^*_k=M^*_k-M^*_{k-1}$, $k\ge 2$, and $\veps^*_1=M^*_1-(1-r)$, it follows, for $n\ge 1$, that
\begin{eqnarray*}
\veps^*_{n+1}&=&\alpha^*_{n+1}\Big(N^*_n\Big(1-\frac{\alpha^*_n}{\alpha^*_{n+1}}\Big)+ I^*_{n+1}\Big)
=\alpha^*_{n+1}\Big(N^*_n\Big(1-\frac{n+1-r}{n}\Big)+ I^*_{n+1}\Big)\\
&=&\alpha^*_{n+1}\Big( I^*_{n+1}- \frac{1-r}{n} N^*_n\Big)\,.
\end{eqnarray*}
The fact that the increments $\veps^*_k$ of the martingale  are centered and uncorrelated tells us that
\begin{eqnarray*}
E((\veps^*_{n+1})^2\, \big|\, {\cal F}_n)&=& (\alpha^*_{n+1})^2\cdot\Big(\frac{(1-r)^2}{n^2}(N^*_n)^2+E(I^*_{n+1}\, \big|\, {\cal F}_n)-2\frac{1-r}{n}N^*_nE(I^*_{n+1}\, \big|\, {\cal F}_n)\Big)\\
&=&(\alpha^*_{n+1})^2\cdot\Big(-\frac{(1-r)^2}{n^2}(N^*_n)^2+\frac{1-r}{n}N^*_n\Big),
\end{eqnarray*}
and that
\begin{eqnarray*}E((\veps^*_{n+1})^2)&=&(\alpha^*_{n+1})^2\cdot
\Big(-\frac{(1-r)^2}{n^2}E((N^*_n)^2)+\frac{1-r}{n}E(N^*_n)\Big)\\
&\sim & \Gamma^2(2-r)\,n^{-2+2r}\frac{(1-r)}{\Gamma(1-r)} n^{-r}=(1-r)^3\cdot
\Gamma(1-r)\, n^{-2+r}\ttt{as}\nifi.
\end{eqnarray*}
For the brackets 
\[<M^*_n> =(\veps^*_1)^2+\sum_{k=2}^n(\alpha^*_{k})^2\cdot
\Big(-\frac{(1-r)^2}{(k-1)^2}(N^*_{k-1})^2+\frac{1-r}{k-1}N^*_{k-1}\Big)\,,\]
so that, with a glance at (\ref{erwartung}) and (\ref{alfastern}),
\[E(<M^*_n>)\leq  C\,\sum_{k=1}^n(\alpha^*_{k})^2\frac{1}{k^{r}}
\leq C\,\sum_{k=1}^\infty\frac{1}{k^{2(1-r)}}\cdot\frac{1}{k^{r}}\leq C<\infty,\]
that is, $E(<M^*_n>)={\cal O}(1)$ as $\nifi$.

Summarizing our findings leads  to the following result.
\begin{theorem}\label{thmx1}There exists a random variable $Y $ such that $M^*_n\asto Y$  as $n \to \infty$, 
and, hence, 
\[\frac{N^*_n}{n^{1-r}}=\frac{n-N_n}{n^{1-r}}\asto \frac{Y}{\Gamma(2-r)} \ttt{as}\nifi.\]
Moreover, convergence holds in $L^1$ in both cases, in particular, 
\[E\Big(\frac{N^*_n}{n^{1-r}}\Big)=E\Big(\frac{n-N_n}{n^{1-r}}\Big)\to \frac1{\Gamma(1-r)}\ttt{as}\nifi.\] 
\end{theorem}
\begin{proof}
As $\{M^*_n\,,n\geq1\}$ is a nonnegative martingale almost sure convergence for the starred process follows via a standard convergence result, see, e.g., p.\ 510 in \cite{g13}. This, together with the fact that $N_n=n- N^*_n$, verifies the first conclusion. Mean convergence is a consequence of the boundedness of the variances. The limit of the expected values is immediate from (\ref{erwartung}). \vsb
\end{proof}
\begin{remark}\emph{(i) It follows from the convergence of the expected values that $E(Y)=1-r$.\\[1mm]
(ii) A similar calculation as above shows that $E((N^*_n/n^{1-r})^3)$ is bounded, and, hence,
\[ \var\Big( \frac{N^*_n}{n^{1-r}}\Big) \to d_r-\frac1{(\Gamma(1-r))^2}\ttt{and} 
\var(Y)=(1-r)^2\Big(\frac{d_r}{(\Gamma(1-r))^2}-1\Big)\,.\] }\vsp\end{remark}

\section{The case $\cfn =\sigma\{X_1\}$}
\setcounter{equation}{0}
This is the easy case. Since every indicator depends on $X_1$ only, it follows that the indicators are independent and indentically distributed. Moreover, $E(I_1^*)= 1-r$ and $E(I_{n+1}^*\mid X_1)=(1-r)\,|X_1|$ for all $n$, which imples that given $I^*_1=1\,,$
$\{N_n^*\,, n\geq1\}$, is a binomial random walk with probability $1-r$ and $\equiv 0$ with probability $r$ and is identically zero if $I^*_1=0.$ 

\begin{theorem}\label{thmx2} In this case\\[1.5mm]
\emph{(a)}  $\hspace*{4mm} \displaystyle{ \frac{N^*_n}{n} \asto (1-r)^2   \ttt{as} \nifi\,,}$\\[2mm]
\emph{(b)}  $\hspace*{4mm} \displaystyle{\frac{N_n^*-n(1-r)I_1^*}{\sqrt{n}}\dto (1-r){\cal N}_{0,r(1-r)}+r\,\, \delta_0(x)\ttt{as}\nifi.}$\\[2mm]
For the number of zeroes we obtain\\[1.5mm]
\emph{(c)}  $\hspace*{4mm} \displaystyle{ \frac{N_n}{n} \asto r\,(2-r) \ttt{as} \nifi\,,}$\\[2mm]
\emph{(d)}  $\hspace*{4mm} \displaystyle{\frac{N_n-nr-n(1-r)I_1}{\sqrt{n}}\dto (1-r){\cal N}_{0,r(1-r)}+r\,\, \delta_0(x)\ttt{as}\nifi.}$
\end{theorem}
\begin{proof} The results for the starred process follows from the arguments preceding the statement. For the nonstarred process we use the additional facts that
\[E(N_n)=n-E(N^*_n)=n-n(1-r)^2 +{\cal O}(1)\sim n\,r\,(2-r)\ttt{as}\nifi,\] 
and that
\[N_n^*-n(1-r)I_1^*=n-N_n -n(1-r)(1-I_1)=-N_n+nr+n(1-r)I_1,
\]
together with the symmetry of the normal distribution.\vsb\end{proof}

\section{The case $\cfn =\sigma\{X_n\}$}
\setcounter{equation}{0}
In \cite{113}, Section 7 we found that there will be almost surely a finite number of non-zero steps. In fact, with $\tau=\min\{n:X_n=0\}$, the event $\{\tau\geq n\}$ occurs precisely if there is no summand that equals to zero among the first $n$ ones. Assuming that $X_1=1$ it follows that $P(\tau\geq n)=(1-r)^{n-1}$,  which shows that
$\tau$ has a geometric distribution with mean $1/r$.

As for the starred counting process it follows from the above that 
\[ P(N_n^*=k)=(I^*_1=1,\ldots,I^*_{k}=1, I^*_{k+1}=0)=(1-r)^{k}r \ttt{for} k=0,\dots, n-1\,,\]
and that $P(N^*_n=n)=P(\tau\geq n+1)=(1-r)^n$.
The generating function therefore turns out as
\[g_{N^*_n}(t)= r \sum_{k=0}^{n-1} (1-r)^{k}t^k+(1-r)^nt^n=\frac{r\,(1-((1-r)t)^{n}}{1-(1-r)t}+(1-r)^nt^n.\]
Hence, for $t<1/(1-r)$,
\[g_{N^*_n}(t)\to \frac{r}{1-(1-r)t} =g(t) \ttt{ as} \nifi\,,\]
which is the generating function of a geometric random variable with mean $(1-r)/r$; in particular, the generating function of $\tau-1$.

The following theorem emerges.
\begin{theorem} \label{thm1x3} \qquad $N^*_n =n-N_n\asto Z \ttt{as}\nifi$,\\[2mm]
where $Z= \tau-1$ is a geometric random variable with  mean $\frac{1-r}{r}$. Moreover, all moments converge.
\end{theorem}
\begin{proof}
Almost sure convergence holds since $N_n^*$ is monotone increasing and bounded almost surely.  
That $Z=\tau-1$ follows from the fact that if the process stops at step $\tau$, then the zero that ends the process is preceded by $\tau-1$ nonzero summands. The convergence of the generating functions implies the rest.\vsb
\end{proof}
\begin{remark}Note that \emph{$\{I^*_n\,, n\geq1\}$ is a two state Markov chain where one state is absorbing.}\vsb
\end{remark}

\section{The case $\cfn =\sigma\{X_1,X_n\}$}
\setcounter{equation}{0}
In this subsection we prove the analog  for the case when the elephant remembers the first and the most recent steps.
\begin{theorem}
In the present situation,\\[1.5mm]
\emph{(a)}  $\hspace*{4mm} \displaystyle{ \frac{N^*_n}{n} \asto \frac{(1-r)^2}{1+r}  \ttt{as} \nifi\,;}$\\[2mm]
\emph{(b)} $\hspace*{4mm}\displaystyle{ P\bigg(\frac{N^*_n-n\cdot\dfrac{1-r}{1+r}\cdot I^*_1}{\sqrt{n}}\leq x \bigg) \dto (1-r){\cal N}_{0,{\sigma^*}^2}(x)\, +r\,\, \delta_0(x)\ttt{as}\nifi.}$\\[2mm]
For the number of zeroes we obtain\\[1.5mm]
\emph{(c)}  $\hspace*{4mm}\displaystyle{\frac{N_n}{n} \asto \frac{r(3-r)}{1+r}\ttt{as} \nifi\,;}$\\[2mm]
\emph{(d)} $\hspace*{4mm}\displaystyle{ P\bigg(\frac{N_n-n\cdot\dfrac{2r}{1+r}-n\cdot\dfrac{1-r}{1+r}\cdot I_1}{\sqrt{n}}\leq x \bigg) \dto (1-r){\cal N}_{0,{\sigma^*}^2}(x)\, +r\,\, \delta_0(x)\ttt{as}\nifi \,,}$  \\[2mm]
where  ${\sigma^*}^2=\frac{6r(1-r)}{(1+r)^2}$.
\end{theorem}
\begin{proof} If $X_1=0$ the random walk stays put at zero. We therefore suppose in the following that $X_1\neq0$, and, hence, that $I_1^*=1$. Then, $E(I_2^*\mid I_1^*)= 1-r=E(I_2^*)$, and, generally, for $n\geq 1$,
\[
E(I_{n+1}^*\mid\cfn)=\frac{1-r}{2}I_1^* + \frac{1-r}{2}I_n^*=\frac{1-r}{2}+\frac{1-r}{2}I_n^*,\]
so that
\[E(I_{n+1}^*)=\frac{1-r}{2}+\frac{1-r}{2}E(I_n^*),\]
which, according to Proposition \ref{diff}, tells us that
\[E(I_n^*)=\frac{(1-r)/2}{1-(1-r)/2}+o(1)=\frac{1-r}{1+r}+o(1)\ttt{as}\nifi.\]
Next, adding the extreme members in the first display yields, for $n \ge 1$,
\[E(N_{n+1}^*)-1=\frac{n(1-r)}{2}+\frac{1-r}{2}E(N_n^*),\]
so that, via Proposition \ref{diff}, 
\bea
E(N_n^*)&=&\frac{(n-1)((1-r)/2)+1}{1-\frac{1-r}{2}}
            -  \frac{((1-r)/2)(n-1)((1-r)/2+1)}{n(1-\frac{1-r}{2})^2}\big(1+o(1)\big)\nonumber\\
 &=& \frac{n(1-r)}{1+r}-\frac{1-r}{1+r}+\frac{2}{1+r}-\frac{(1-r)^2}{(1+r)^2}+o(1)\nonumber\\
&=& \frac{n(1-r)}{1+r}+   \frac{4r}{(1+r)^2}+o(1)\ttt{as}\nifi. \label{egemischt}
\eea   
As for second moments (note that $I^2=I$),
\beaa
E((N_{n+1}^*)^2\mid\cfn)&=&(N_n^*)^2+2N_n^*E(I_{n+1}^*\mid\cfn)+E((I_{n+1}^*)^2\mid\cfn)\\[2mm]
&=&(N_n^*)^2+2N_n^*\Big(\frac{1-r}{2}+\frac{1-r}{2}I_n^*\Big) +\frac{1-r}{2}+\frac{1-r}{2}I_n^*,
\eeaa
and, hence,
\bea
E((N_{n+1}^*)^2)&=&E((N_{n}^*)^2)+(1-r)E(N_n^*)+(1-r)E(N_n^*I_n^*)+\frac{1-r}{2}\nonumber\\
&&\hskip2pc+\,\frac{1-r}{2}\Big(\frac{1-r}{1+r}+o(1)\Big)\nonumber\\
&=&E((N_{n}^*)^2)+(1-r)E(N_n^*)+(1-r)E(N_n^*I_n^*)\nonumber\\
&&\hskip2pc+\,\frac{1-r}{1+r}+o(1)\ttt{as}\nifi.
\label{momzprel}\eea
In order to continue we need the mixed moment:
\beaa
E(N_{n+1}^*I_{n+1}^*\mid\cfn)&=&N_n^*E(I_{n+1}^*\mid\cfn)+E((I_{n+1}^*)^2\mid\cfn)\\[2mm]
&=&N_n^*\Big(\frac{1-r}{2}+\frac{1-r}{2}I_n^*\Big)+\frac{1-r}{2}+\frac{1-r}{2}I_n^*,
\eeaa
so that
\beaa
E(N_{n+1}^*I_{n+1}^*)&=&\frac{1-r}{2}E(N_{n}^*I_{n}^*)+\frac{1-r}{2}E(N_n^*)
+\frac{1-r}{2}+\frac{1-r}{2}E(I_n^*)\\[1mm]
&=&\frac{1-r}{2}E(N_{n}^*I_{n}^*)+\frac{n(1-r)^2}{2(1+r)} +\frac{2r(1-r)}{(1+r)^2}
+\frac{1-r}{2}+\frac{(1-r)^2}{2(1+r)}\\[1mm]
&=&\frac{1-r}{2}E(N_{n}^*I_{n}^*)+\frac{n(1-r)^2}{2(1+r)}+\frac{(1-r)(1+3r)}{(1+r)^2}.
\eeaa
An application of Proposition \ref{diff}(i) shows that
\beaa
E(N_{n}^*I_{n}^*)&=&\sum_{k=0}^{n-2} \Big(\frac{1-r}{2}\Big)^k\cdot\frac{(n-1-k)(1-r)^2}{2(1+r)}
+\frac{(1-r)(1+3r)}{(1+r)^2}\sum_{k=0}^{n-2}\Big(\frac{1-r}{2}\Big)^k +o(1)\\
&=&\frac{(n-1)(1-r)^2}{2(1+r)}\sum_{k=0}^{n-2}\Big(\frac{1-r}{2}\Big)^k
-\frac{(1-r)^2}{2(1+r)}\sum_{k=0}^{n-2}k\,\Big(\frac{1-r}{2}\Big)^k\\
&&\hskip2pc+\, \frac{2(1-r)(1+3r)}{(1+r)^3}+o(1)\nonumber\\
&=& \frac{n(1-r)^2}{(1+r)^2} -\frac{(1-r)^2}{(1+r)^2}- \frac{(1-r)^3}{(1+r)^3}+ \frac{2(1-r)(1+3r)}{(1+r)^3}+o(1)\nonumber\\
&=&n\cdot\frac{(1-r)^2}{(1+r)^2}+\frac{8r(1-r)^2}{(1+r)^3}+o(1)\ttt{as}\nifi.\label{gemischtes}
\eeaa
Inserting this into (\ref{momzprel})  shows that
\beaa
E((N_{n+1}^*)^2)&=&E((N_{n}^*)^2)+(1-r)\Big(n\cdot\frac{1-r}{1+r}+\frac{4r}{(1+r)^2}\Big)\\
&&\hskip2pc+\,(1-r)\Big(n\cdot\frac{(1-r)^2}{(1+r)^2}+\frac{8r(1-r)^2}{(1+r)^3}\Big)+\frac{1-r}{1+r}+o(1)\\
&=&E((N_{n}^*)^2) +n\cdot\frac{2(1-r)^2}{(1+r)^2}+\frac{(1-r)(1+14r-3r^2)}{(1+r)^3}+o(1)\ttt{as}\nifi,
\eeaa
after which telescoping yields
\bea
E((N_{n}^*)^2)&=&\frac{n(n-1)}{2}\cdot\frac{2(1-r)^2}{(1+r)^2} +n\cdot\frac{(1-r)(1+14r-3r^2)}{(1+r)^3}
+o(n)\nonumber\\
&=& n^2\cdot\frac{(1-r)^2}{(1+r)^2} + n\cdot\frac{2r(1-r)(7-r)}{(1+r)^3}+o(n)\ttt{as}\nifi.\label{zweites}
\eea
Putting things together, finally, leads to the variance
\bea\label{mischvarianz}
\var(N_n^*)&=& n^2\cdot\frac{(1-r)^2}{(1+r)^2} + n\cdot\frac{2r(1-r)(7-r)}{(1+r)^3}+o(n)
-\Big(\frac{n(1-r)}{1+r}+   \frac{4r}{(1+r)^2}+o(1)\Big)^2\nonumber\\
&=& n\cdot\Big(\frac{2r(1-r)(7-r)}{(1+r)^3}-2\frac{4r(1-r)}{(1+r)^3}\Big)+o(n)\nonumber\\
&=&n\cdot\frac{6r(1-r)(1+r)}{(1+r)^3}+o(n)= n\cdot\frac{6r(1-r)}{(1+r)^2}+o(n)\nonumber\\
&=&n\cdot{\sigma^*}^2 +o(n)\ttt{as}\nifi.\label{varmisch}
\eea
Being in the branch with $I^*_1=1$ we are faced with a stationary ergodic Markov chain, which asserts the validity of  the  strong law, and, via  conditioning on the two branches, the central limit theorem for the starred sequence. 

Noticing that
\[N_n^* -n\cdot\frac{1-r}{1+r}I_1^*=n-N_n-n\cdot\frac{1-r}{1+r}(1-I_1)
=-\Big(N_n-n\cdot\frac{2r}{1+r}-n\cdot\frac{1-r}{1+r}I_1\Big),
\]
that the variances of the two counting processes are the same and that the standard normal distribution is symmetric,
establishes the second part.\vsb
\end{proof}

\section{Remarks}
\begin{remark}\emph{(i) When $\cfn=\sigma \{X_{n-1},X_n\}$, then $E(I^*_n)$ again tends to zero geometrically fast, and, hence, $I^*_n\neq0$ only finitely often w.p.1. Similarly in the case when $\cfn=\sigma \{X_{n-m},\ldots,X_n\}$ for some $m \in \mathbb{N}$.\\[1mm]
(ii) Whereas the correlation Cor$(I_n^*, I_{n+1}^*)={\cal O}(n^{-r})$ as $\nifi$ when 
$\{\cfn=\sigma \{X_1,\ldots,X_n\}$, the  correlation tends to a nonvanishing limit in all other cases.\\[1mm]
(iii) In the cases considered above the sequences $\{N^*_n\,,n\geq1\}$ show different asymptotic behaviors. It would be of interest to learn more about the change points.\\[1mm]
(iv) Related to (iii) and as was mentioned in our two predecessors, \cite{111, 113},  one might think of cases where the length of the memory depends on $n$, typically $\log n$ or some power of $n$. \\[1mm]
(v) The asymptotics fo the ERW obviously depend on $p$ and $q$. This is equally obviously not the case for the counting processes in this sequel, since the only point of interest here is whether there is a zero or not. This is why all results and computations only involve $r$. In other words, $p$ and $q$ may vary along the way as long as their sum, $r$, remains constant.}\vsb
\end{remark}

\subsection*{Acknowledgement} We wish to thank Svante Janson for his remark ''det blir en v\"aldig massa nollor'' (there will be a lot of zeroes) when one of us (A.G.) introduced him to the world of \erw s.

\medskip\noindent {\small Allan Gut, Department of Mathematics,
Uppsala University, Box 480, SE-751\,06 Uppsala, Sweden;\\
Email:\quad \texttt{allan.gut@math.uu.se}\\
URL:\quad \texttt{http://www.math.uu.se/\~{}allan}}
\\[4pt]
{\small Ulrich Stadtm\"uller, Ulm University, Department of Number
Theory
and Probability Theory,\\ D-89069 Ulm, Germany;\\
Email:\quad \texttt{ulrich.stadtmueller@uni-ulm.de}\\
URL:\quad
\texttt{http://www.mathematik.uni-ulm.de/en/mawi/institute-of-number-theory-and-probability-\\theory/people/stadtmueller.html}}

\end{document}